\documentclass{amsart}
\usepackage{amsmath,amssymb,amsthm,mathrsfs, dcpic, pictexwd}
\usepackage[margin=1in]{geometry}
\begin{document}
\fontsize{10.95}{14}\rm
\newcommand{\oper}[1]{\operatorname{#1}} 
\newcommand{\p}{\mathbb{P}^1}
\newcommand{\z}{\mathbb{Z}}
\newcommand{\xx}{\mathcal{X}}
\newcommand{\yy}{\mathcal{Y}}
\newcommand{\co}{\mathbb{C}}
\newcommand{\pp}{\mathscr{P}}
\newcommand{\f}{\mathbb{F}_p}
\newcommand{\fc}{\overline{\mathbb{F}_p}}
\newcommand{\iso}{\to^{\!\!\!\!\!\!\!\sim\,}}
\newcommand{\q}{\mathbb{Q}}
\newcommand{\tm}{((t^{\frac{1}{m}}))}
\newcommand{\zun}{\z_p^{\oper{un}}}
\newcommand{\qun}{\q_p^{\oper{un}}}
\newcommand{\kar}{\overline{K}}
\newcommand{\pa}{\mathfrak{p}}
\newcommand{\al}{\widetilde{\mathbb{Z}}^{\overline{\mathbb{Q}}}}
\newcommand{\lal}{\widetilde{\mathbb{Z}_p}^{\overline{\mathbb{Q}_p}}}
\title{Minimal Fields of Definition for Galois Action}
\author{Hilaf Hasson}
\date{September 9, 2015}
\begin{abstract}
Let $K$ be a field, let $G$ be a finite group, and let $\bar X\rightarrow \bar Y$ be a $G$-Galois branched cover of varieties over $K^{\oper{sep}}$.
Given a mere cover model $X\rightarrow Y$ of this cover over $K$, in Part I of this paper I observe that there is a unique minimal field $E$ over which
$X\rightarrow Y$ becomes Galois, and I prove that $E/K$ is Galois with group a subgroup of $\oper{Aut}(G)$. In Part II of this paper, by making the additional
assumption that $K$ is a field of definition (i.e., that there exists \it some \rm Galois model over $K$), I am able to give an explicit
description of the unique minimal field of Galois action for $X\rightarrow Y$. Namely, if there exists a $K$-rational point of $X$ above an unramified point $P\in Y(K)$
then $E$ is contained in the intersection of the specializations at $P$ in the various different $G$-Galois models of $\bar X\rightarrow \bar Y$ over $K$.
Using the same proof mechanism, I observe a reverse version of ``The Twisting Lemma'', which asserts that the behavior of the $K$-rational 
points on the various mere cover models over $K$, and the behavior of the specializations on the various $G$-Galois models 
over $K$, are all governed by a single equivalence 
relation (independent of the model) on the 
$K$-rational points of the base variety.
\\ \\
MSC classes:	14H30, 14G05, 11S20, 12F12
\end{abstract}
\maketitle
\theoremstyle{plain}
\newtheorem{thm}{Theorem}[section]
\newtheorem{rmk}[thm]{Remark}
\newtheorem{qst}[thm]{Question}
\newtheorem{prp}[thm]{Proposition}
\newtheorem{hyp}[thm]{Hypotheses}
\newtheorem{crl}[thm]{Corollary}
\newtheorem{cnv}[thm]{Convention}
\newtheorem{cnj}[thm]{Conjecture}
\newtheorem{stt}[thm]{Statement}
\newtheorem{lem}[thm]{Lemma}
\newtheorem{dff}[thm]{Definition}
\newtheorem{clm}[thm]{Claim}
\newtheorem{ntt}[thm]{Notation}
\newtheorem{example}[thm]{Example}
\section{Introduction}
The focus of this paper is the descent theory of algebraic covers, and more precisely of $G$-Galois branched covers. For a finite group 
$G$, a map of varieties (over a fixed field) is said to be a \it $G$-Galois branched cover \rm if it is finite, generically \'etale, and 
$G$ acts freely and transitively on its geometric fibers, away from the ramification.

Questions concerning the descent behavior of $G$-Galois 
branched covers  
arise naturally in Arithmetic Geometry and Galois Theory. For example, an easy corollary of Riemann's Existence Theorem (\cite{sgaone}, expos\'e XII) is
that for every finite group $G$ there exists a $G$-Galois branched cover over $\p_{\bar \q}$. By
Hilbert's Irreducibility Theorem (\cite{fj}, Chapter 11), if this cover descends (together with its group action) to a number field $K$, then $G$ is realizable as the Galois group of a field extension over $K$.

Every $G$-Galois branched cover has an associated field called the \it field of moduli \rm (see Definition \ref{E:fomfom}), 
which is the best candidate for being the unique 
minimal field of 
definition of the cover (together with its Galois action), if one exists. Questions 
regarding the field of moduli have been a vibrant area of research (see for example \cite{syb1}, \cite{raynaud} and \cite{vs}). 

The field of moduli is contained in every field of definition of the $G$-Galois branched cover. Furthermore, David Harbater and Kevin Coombes proved in \cite{ch}
that (assuming the $G$-Galois branched cover is given over $\bar \q$) it is equal to the intersection of all fields of definition. It is important to note that while the field of moduli
may not be a field of definition of the cover \it together with its Galois action\rm, it was proven in \cite{ch} that it must be a field of definition
of the cover as a mere cover. (For more about the distinction between ``field of definition as a $G$-Galois branched cover'' and ``field of definition as
a mere-cover'', see Definition \ref{E:models}.)

These two results motivate the question studied in this paper. Namely, let $K$ be a field, let $\bar X\rightarrow \bar Y$ be a $G$-Galois branched cover defined over $K^{\oper{sep}}$,
and let $X\rightarrow Y$ be a mere-cover model (i.e., the model is not required to be Galois) of it over $K$. We ask the question: what can be said
about the minimal fields (or, as it turns out, field) $L$ that contain $K$ such that $X\times_KL\rightarrow Y\times_KL$ is Galois? In view of 
the results in \cite{ch} cited above, the answer to this
question informs our understanding of the relationship between the field of moduli of the $G$-Galois branched cover, and its minimal fields of definition.
(For further discussion see Remark \ref{E:explanation}.)

In the situation above, let $E$ be the intersection of all minimal fields of Galois action for the mere-cover model $X\rightarrow Y$. 
In Part I of this paper (specifically, Theorem \ref{E:tech}), I observe that $E$ is itself a field of Galois action of $X\rightarrow Y$, and I prove that $E/K$ is Galois with automorphism group a subgroup of $\oper{Aut}(G)$.
As a corollary (Corollary \ref{E:thebigone}) of this result, we see that the field of moduli as a $G$-Galois branched cover is Galois over the field of moduli as a mere-cover,
with group a subquotient of $\oper{Aut}(G)$.

Part II of this paper 
continues to explore the unique minimal field of Galois action $E$ of a particular mere cover model $X\rightarrow Y$ of $\bar X\rightarrow \bar Y$, but under the additional assumption that $K$ is a
field of definition (together with the Galois action) of $\bar X\rightarrow \bar Y$. (In other words, in addition to the assumptions of Part I,
we assume that there exists \it some \rm $G$-Galois model of $\bar X\rightarrow \bar Y$ over $K$.)
Under this additional assumption, we are able to deduce a lot more about $E$. Namely,
Theorem \ref{E:crux} says that if
there's a $K$-rational point of $X$ above an unramified point $P\in Y(K)$, then $E$ is contained in the the intersection of the fields induced by 
specializing at $P$ in
all of the different $G$-Galois models of $\bar X\rightarrow \bar Y$ over $K$. (If the fiber over $P$ is disconnected then by ``the field induced by specializing at $P$''
I mean the Galois closure of any of its connected components; see Remark \ref{E:rmkofspec} for further details.)

The same mechanism that proves the main result of Part II also gives a result (Theorem \ref{E:specialization}) 
that describes the 
behavior of the $K$-rational points on the various mere cover models over $K$, as well as the behavior of the specializations of 
the various $G$-Galois models over $K$, by a single equivalence relation 
(independent of the model) on the $K$-rational 
points of the base variety. One can view this result as a reverse version of ``The Twisting Lemma'' (Lemma \ref{E:twistinglemma}; see also \cite{debesgrunwald}).

While Parts I and II of this paper explore the minimal fields of Galois action of a given mere cover model, and therefore inform our understanding of minimal fields
of definition over the field of moduli (see Remark \ref{E:explanation}), in the appendix we stregthen a result that appeared in \cite{ch}, and 
construct a special field of definition
(infinite over the field of moduli) for every $G$-Galois branched cover. As a
corollary of this result, we prove that for every finite group $G$ there exists an extension of number fields $\q\subset E\subset F$ such that $F/E$
is $G$-Galois, and $E/\q$ ramifies only over those primes that divide $|G|$.
\section{Definitions and Notations}\label{E:sectdefinitions}
\begin{ntt}\rm
 Let $X$ be an integral scheme. We will use the notation $\kappa(X)$ to denote the function field of $X$.
\end{ntt}
\begin{ntt}\rm
 For every field $K$ we will use the notation $K^{\oper{sep}}$ to denote its separable closure.
\end{ntt}
\begin{dff} \label{E:defofgalois}\rm
Let $X$ and $Y$ be integral schemes. Assume that $X$ is normal, and $Y$ is regular. We say that a map $X \rightarrow Y$ is a \it branched
cover \rm (or simply a cover) if the map is finite and generically \'etale. Let $B$ be the codimension $1$ reduced induced subscheme of $Y$ 
made up of the branching locus of this map. (We will henceforth use the term \it the branch divisor \rm to refer to this construction.) We 
say 
that the
branched cover is \it Galois \rm if $\oper{Aut}(X/Y)$ acts freely and transitively on the geometric fibers over $Y\smallsetminus 
B$. 
We
sometimes refer to branched covers as
\it mere covers\rm.

Let $G$ be a finite group. A \it $G$-Galois branched cover \rm is a pair $(X\rightarrow Y,\Phi)$, where $X\rightarrow Y$ is a 
branched cover that is Galois, and $\Phi$ is an isomorphism from $\oper{Gal}(\kappa(X)/\kappa(Y))$ 
to $G$. (We will sometimes abuse notation and say that $X\rightarrow Y$ is a $G$-Galois branched cover without mentioning its associated 
isomorphism.)

Let $X\rightarrow Y$ and $X'\rightarrow Y$ be two mere covers of $Y$. We say that they 
are \it isomorphic as mere covers \rm if there exists an isomorphism $\eta:X\rightarrow X'$ that makes the following commute:
$$\begindc{\commdiag}[450]
\obj(0,1)[do]{$X$}
\obj(1,0)[dt]{$Y$}
\obj(2,1)[df]{$X'$}
\mor{do}{dt}{}
\mor{do}{df}{$\eta$}
\mor{df}{dt}{}
\enddc.$$

If the two mere covers are $G$-Galois branched covers, then we say that they are \it isomorphic as $G$-Galois branched covers \rm if there 
exists an isomorphism $\eta$ 
as above that agrees with the given isomorphisms of $\oper{Gal}(\kappa(X)/\kappa(Y))$ and $\oper{Gal}(\kappa(X')/\kappa(Y))$ with 
$G$.

 Let $L$ and $K$ be fields, and let $G$ be a finite group. If $\oper{Spec}(L)\rightarrow \oper{Spec}(K)$ is a mere cover (resp. $G$-Galois 
cover), then we say that $L/K$ is a \it mere extension of fields \rm (resp. a \it $G$-Galois extension of fields\rm ).
\end{dff}

Both the hypotheses of Part I (Hypotheses \ref{E:weaaak}) and the hypotheses of Part II (Hypotheses \ref{E:hyp1}) will be based on the 
following setting.
\begin{hyp}\label{E:superweaaak}\rm
 Let $K$ be a field, let $G$ be a finite group, and let $(\bar X\rightarrow \bar Y, \Phi)$ be a
$G$-Galois branched cover of normal, geometrically irreducible
varieties over $K^{\oper{sep}}$, where $\bar Y$ is 
smooth. Assume  
$\bar Y$ descends to $K$, and let $Y$ be a $K$-model of it. (Capital letters with a bar over them will consistently denote varieties over separably closed fields.)
\end{hyp}
The remainder of this section will refer to the situation of Hypotheses \ref{E:superweaaak}.
\begin{dff}\label{E:models}\rm
 Let $E$ be a subfield of $K^{\oper{sep}}$ that contains $K$. Then we say that $E$ is a \it field of definition of 
$\bar X\rightarrow \bar Y$ as a mere cover \rm if it 
descends to a map of
$E$-varieties $X\rightarrow Y\times_KE(=:Y_E)$. (Any such $X\rightarrow Y_E$ is called an $E$-model, or a mere cover model over $E$, 
of $\bar X\rightarrow \bar Y$.) We say
that $E$ is a \it field of definition as a $G$-Galois branched cover \rm if $\bar X \rightarrow \bar Y$ has an 
$E$-model that is
Galois. (Any such $E$-model is called a $G$-Galois model of $\bar X \rightarrow \bar Y$ over $E$.)
\end{dff}
\begin{ntt}\rm
Let $X\rightarrow Y_E$ be a $G$-Galois model of $(\bar X\rightarrow \bar Y,\Phi)$. Note 
 that $\oper{Gal}(\kappa(X)/\kappa(Y_E))$ is naturally isomorphic to $\oper{Gal}(\kappa(\bar X)/\kappa(\bar Y))$, and 
that 
therefore 
$X\rightarrow Y_E$ comes equipped with the structure of a $G$-Galois branched cover. In order to simplify notation we will 
denote 
the induced isomorphism from $\oper{Gal}(\kappa(X)/\kappa(Y_E))$ to $G$ also by $\Phi$.
\end{ntt}
\begin{dff}\label{E:fomfom}\rm
The \it field of moduli \rm of $\bar X\rightarrow\bar Y$ as a $G$-Galois branched (resp. mere cover) is the subfield 
of $K^{\oper{sep}}$ fixed under the subgroup of $\oper{Gal}(K^{\oper{sep}}/K)$ made up of the automorphisms $\sigma$ that take 
$\bar X \rightarrow \bar Y$ to an isomorphic copy 
of itself as a $G$-Galois branched cover (resp. mere cover):
$$\begindc{\commdiag}[450]
\obj(0,1)[do]{$\bar X^{\sigma}$}
\obj(1,0)[dt]{$Y\times_K K^{\oper{sep}}(=\bar Y)$}
\obj(2,1)[df]{$\bar X$}
\mor{do}{dt}{}
\mor{do}{df}{$\sim$}
\mor{df}{dt}{}
\enddc.$$
\end{dff}
\section{Part I - $K$ is a Field of Definition as a Mere Cover}\label{E:sectweak}
Throughout Part I we will often refer back to the following set of hypotheses.
\begin{hyp}\label{E:weaaak}\rm
 The same as Hypotheses \ref{E:superweaaak}, with the extra assumption that $K$ is a field of definition of $\bar X\rightarrow 
\bar Y$ as a mere cover.
\end{hyp}
\begin{thm} \label{E:tech}
In the situation of Hypotheses \ref{E:weaaak}, let
$X\rightarrow
Y$ be a mere cover model of $\bar X\rightarrow \bar Y$ over $K$, and let $E$ be the intersection of all
of the overfields $L$ of $K$ such that the base change $X\times_K L \rightarrow Y\times_KL$ is Galois. Then the following hold:
\begin{enumerate}
\item The base change of $X\rightarrow Y$ to $E$ is already Galois.
\item The field extension $E/K$ is Galois, with Galois group isomorphic to a
subgroup $H$ of $\oper{Aut}(G)$.
\item The field extension $\kappa(X_E)/\kappa(Y)$ is Galois, with group isomorphic to $G\rtimes H$ 
(where \\$\oper{Gal}(\kappa(X_E)/\kappa(Y_E))\cong G$ is the
obvious subgroup of $G\rtimes H$, and where the action of $H$ on $G$ is
given by the embedding of $H$ in $\oper{Aut}(G)$).
\item $E$ is the intersection of $K^{\oper{sep}}$ with the Galois closure of $\kappa(X)/\kappa(Y)$.
\end{enumerate}
\end{thm}

It is well known, by Grothendieck's theory of faithfully flat descent (\cite{grdesc}), 
that the set of mere cover models over $K$ of $\bar X\rightarrow \bar Y$ that lie over $Y$, up to isomorphism of mere cover models, is in bijection with the set $H^1(\oper{Gal}(K^{\oper{sep}}/K),G)$.
In order to prove Theorem \ref{E:tech}, we will look instead at elements in $Z^1(\oper{Gal}(K^{\oper{sep}}/K),G)$, which can be identified with the sections of a certain short exact sequence.
The key to proving Theorem \ref{E:tech} is to study the relationship between properties of elements in $Z^1(\oper{Gal}(K^{\oper{sep}}/K),G)$
and their respective mere cover models (Lemma \ref{E:modelssections}). We proceed now to make the above identifications explicit.

In the situation of Hypotheses \ref{E:weaaak}, we have following diagram of fields:
$$\begindc{\commdiag}[250]
\obj(0,3)[do]{$\kappa(\bar Y)$}
\obj(0,1)[dt]{$K^{\oper{sep}}$}
\obj(3,2)[df]{$\kappa(Y)$}
\obj(3,0)[zz]{$K$}
\obj(0,5)[za]{$\kappa(\bar X)$}
\mor{do}{dt}{}[\atright,\solidline]
\mor{do}{df}{}[\atright,\solidline]
\mor{df}{zz}{}[\atright,\solidline]
\mor{zz}{dt}{}[\atright,\solidline]
\mor{do}{za}{$G$}[\atleft,\solidline]
\mor{za}{df}{}[\atright,\solidline]
\enddc$$

Since we assumed that $K$ is a field of definition as a mere cover, it is in particular the field of moduli as mere cover, 
which immediately implies that $\kappa(\bar X)$ is
Galois over $\kappa(Y)$. (Compare with Lemma 2.4 in 
\cite{syb2}; see also \cite{matzat}.)

We, therefore, have a short exact sequence:
$$1\rightarrow G \rightarrow \oper{Gal}(\kappa(\bar X)/\kappa(Y))\xrightarrow[]{f} \oper{Gal}(K^{\oper{sep}}/K)\rightarrow 1$$

The set of sections of $f$ is in bijection with $Z^1(\oper{Gal}(K^{\oper{sep}}/K),G)$, and therefore each section
induces a mere cover model. To be explicit, let $s$ be a section of $f$, and let $F$ be the subfield of $\kappa(\bar X)$ fixed by $s(\oper{Gal}(K^{\oper{sep}}/K))$.
Since $G\cap s(\oper{Gal}(K^{\oper{sep}}/K))=1$, the compositum $F\cdot \kappa(\bar Y)$ is equal to $\kappa(\bar X)$.
Furthermore, $K$ is algebraically closed in $F$.
(In order to see this, note that it is straightforward to see that the
natural map $G\rightarrow\oper{Gal}(\kappa(\bar X)/\kappa(Y))/s(\oper{Gal}(K^{\oper{sep}}/K))$ is a bijection of 
sets. Therefore the 
field
$F$ has 
degree $|G|$ over $\kappa(Y)$. Since $F\cdot \kappa(\bar Y)=\kappa(\bar X)$, the field $\kappa(\bar Y)$ is linearly disjoint from 
$F$ over $\kappa(Y)$.) Therefore the normalization $X$ of $Y$ in $F$ gives a mere-cover model $X\rightarrow Y$ of $\bar X\rightarrow \bar Y$ over $K$.

It should be noted that different sections of $f$ may induce isomorphic mere-cover models. Therefore, we may wish to consider a slightly more coarse category.
\begin{lem}\label{E:becky}
In the situation of Hypotheses \ref{E:weaaak}, let $$\Omega:=\{F|\kappa(Y)\subset F\subset\kappa(\bar X), \mbox{\,such that\,\,\,}F\cdot\kappa(\bar Y)=\kappa(\bar X) \mbox{,\,and\,}K \mbox{\,is algebraically closed  in\,} F\},$$
and let $\Psi:\oper{Sec}(f)\rightarrow \Omega$ be the map taking a section $s$ to the fixed subfield of $\kappa(\bar Y)$ under $s(\oper{Gal}(K^{\oper{sep}}/K))$. Then $\Psi$ is a bijection.
\end{lem}
\begin{proof}
 This is a standard argument, and the proof
is a straightforward generalization of Lemma 2.2.5 in \cite{beckthesis} (where this correspondence is stated in the particular
case that $Y$ is the projective line over a number field).
\end{proof}

The following lemma gives a property for sections of $f$ that corresponds to a mere-cover model being Galois.

\begin{lem} \label{E:modelssections}
 In the situation of Lemma \ref{E:becky}, let $X\rightarrow
Y$ be a mere cover model of $\bar X\rightarrow \bar Y$ over $K$. Then the following are equivalent:
\begin{enumerate}
 \item $X\rightarrow Y$ is Galois.
 \item There exists a $\kappa(Y)$-embedding $\rho$ of $\kappa(X)$ into $\kappa(\bar X)$ such that the 
image of
$\Psi^{-1}(\rho(\kappa(X)))$ commutes
with $G$.
 \item For every $\kappa(Y)$-embedding $\rho$ of $\kappa(X)$ into $\kappa(\bar X)$, the 
image of
$\Psi^{-1}(\rho(\kappa(X)))$ commutes
with $G$.
\end{enumerate}
\end{lem}
\begin{proof}
Let $\rho$ be a $\kappa(Y)$-embedding of $\kappa(X)$ into $\kappa(\bar X)$. Let $s=\Psi^{-1}(\rho(\kappa(X)))$. It suffices to show that the image of
$s$ commutes
with $G$ if and only if $X\rightarrow Y$ is Galois. By definition $X\rightarrow Y$ is Galois if and only if $\kappa(X)/\kappa(Y)$
is a Galois extension of fields, which, in turn, holds if and only if $\rho(\kappa(X))/\kappa(Y)$ is a Galois extension.
By Lemma \ref{E:becky}, $\rho(\kappa(X))$ is the fixed subfield of $\kappa(\bar Y)$ under the image of $s$.
Therefore, by Galois Theory, the extension $\rho(\kappa(X))/\kappa(Y)$ is Galois exactly when the image of $s$ is normal
in $\oper{Gal}(\kappa(\bar X)/\kappa(Y))$. Since $\oper{Gal}(\kappa(\bar X)/\kappa(Y))$ is the semi-direct product
of $G$ with the image of $s$, this condition is equivalent to the image of $s$
commuting with $G$.
\end{proof}

In order
to prove Theorem \ref{E:tech} we require a group-theoretic lemma (Lemma \ref{E:group}).
\begin{ntt}
\rm Let $g$ and $h$ be elements in a group $G$. We use the notation $^hg$ to mean the conjugation $hgh^{-1}$.
\end{ntt}
\begin{lem} \label{E:group}
 Let $J$ and $M$ be groups, and let $I$ be a semi-direct product
$J\rtimes M$. Let $N$
be $M\cap C_I(J)$, where $C_I(J)$ is the centralizer of $J$ in $I$. Then the following hold:
\begin{enumerate}
 \item $N$ is normal in $I$.
 \item Let $\gamma:M/N\rightarrow \oper{Aut}(J)$ be defined by taking $mN$
to the automorphism $j\mapsto \,^mj$. Then $\gamma$ is well defined and injective.
 \item $I/N$ is isomorphic to  the semi-direct product $J\rtimes_{\gamma}
(M/N)$.
\end{enumerate}
\end{lem}
\begin{proof}
 Since $J$ is normal in $I$, it follows that so is $C_I(J)$. Therefore $N$ is normal in $M$. In order to
show that $N$ is
normal in $I$ it suffices to prove for every $n$ in $N$, $j$ in $J$,
and $m$ in $M$ the element $^{jm}n$ is in $N$. Since $N$
is normal in $M$, the element $^mn$ is in $N$. Since $J$
commutes with $N$ it follows that $^{jm}n=\,^j(^mn)=\,^mn$. It is now clear
that $^{jm}n$ is in $N$, and therefore (1) is proven.

The homomorphism $\gamma$ is well defined because $N$ commutes with $J$. It
remains to show that $\gamma$ is injective. Indeed if
$\gamma(mN)=id$ then for every $j\in J$, we have
$^mj=j$. Therefore $m$ commutes with $J$. Since
$m$ is also in $M$, we conclude that it is in $N$. Therefore
$mN=N$. This proves (2).

It is now an easy verification that the map $I=J\rtimes M\rightarrow
J\rtimes_{\gamma}(M/N)$ taking $jm$, where $j\in J$ and $m\in M$, to $(j,mN)$ is a well-defined homomorphism with kernel $N$, proving (3).
\end{proof}
We are now ready to prove Theorem \ref{E:tech}:
\begin{proof} (Theorem \ref{E:tech})
Let $\rho$ be a $\kappa(Y)$-embedding of $\kappa(X)$ into $\kappa(\bar X)$, and let $s=\Psi^{-1}(\rho(\kappa(X)))$ be
the corresponding section of $f$.

Let $V$ be the intersection of
$s(\oper{Gal}(K^{\oper{sep}}/K))=\oper{Gal}(\kappa(\bar X)/\rho(\kappa(X)))$ with
the centralizer of $G$ in $\oper{Gal}(\kappa(\bar X)/\kappa(Y))$. 
Applying Lemma \ref{E:group} with $G$ in the role of $J$, $s(\oper{Gal}(K^{\oper{sep}}/K))$
in the role of $M$, $V$ in the role of $N$, and
$\oper{Gal}(\kappa(\bar X)/\kappa(Y))$ in the role of $I$, we see that $V$ is normal in 
$\oper{Gal}(\kappa(\bar X)/\kappa(Y))$, and that
$\oper{Gal}(\kappa(\bar X)/\kappa(Y))/V$ is isomorphic to a semi-direct product of $G$ with a subgroup of $\oper{Aut}(G)$. In 
particular,
the
group $V$ has finite index in 
$\oper{Gal}(\kappa(\bar X)/\kappa(Y))$, and therefore so does the compositum $GV$. Since $GV$ contains $G$, 
there exists a finite field extension $E'$ 
of $K$,
contained in
$K^{\oper{sep}}$, such
that the
fixed
subfield of $\kappa(\bar X)$ by $GV$ is equal to $\kappa(Y_{E'})$. Note that
$E'\rho(\kappa(X))$ is
the fixed subfield of $\kappa(\bar X)$ by $V$, and therefore we may extend $\rho$ to an embedding of $\kappa(X_{E'})$ into
$\kappa(\bar X)$ with image $E'\rho(\kappa(X))$.

By Lemma
\ref{E:modelssections} (using the extended embedding $\rho$), the map $X_{E'}\rightarrow Y_{E'}$ is Galois because
the image of the restriction of $s$ to $\oper{Gal}(K^{\oper{sep}}/E')$ commutes
with $G$. If $L$ is any finite separable extension of $K$ for which $X_L\rightarrow Y_L$ is Galois, then again by Lemma
\ref{E:modelssections}, the image of the restriction of $s$
to $\oper{Gal}(K^{\oper{sep}}/L)$ commutes with $G$. But this implies that $\oper{Gal}(\kappa(\bar X)/\kappa(Y_L))$ is contained in $GV$. 
Therefore $E'$ is contained in all such $L$, and is in fact the unique minimal one. Therefore $E'$ is equal to $E$, which concludes the proof of assertion (1).

The group $\oper{Gal}(E/K)\cong\oper{Gal}(\kappa(Y_E)/\kappa(Y))$ is isomorphic $s(\oper{Gal}(K^{\oper{sep}}/K))/V$ by the second 
isomorphism 
theorem. It follows from
the above that
$\oper{Gal}(E/K)$ embeds into $\oper{Aut}(G)$. This proves assertion (2) of Theorem \ref{E:tech}. 

Assertion (3) of Lemma \ref{E:group}, applied to
our situation as above, implies that
the field extension $\kappa(X_E)/\kappa(Y)$ is Galois
with Galois group isomorphic to $G\rtimes H$ (where the action of $H$ on $G$ is
given by the embedding of $H$ in $\oper{Aut}(G)$); and that furthermore, we have
$\kappa(X_E)^G=\kappa(Y_E)$, thus proving assertion (3) of Theorem \ref{E:tech}.

In order to prove assertion (4) of Theorem \ref{E:tech}, it suffices to prove that $V$ is the normal core of 
$s(\oper{Gal}(K^{\oper{sep}}/K))$ in $\oper{Gal}(\kappa(\bar X)/\kappa(Y))$ (i.e., the intersection of the groups 
$\sigma^{-1}s(\oper{Gal}(K^{\oper{sep}}/K))\sigma$ as $\sigma$ goes over $\oper{Gal}(\kappa(\bar X)/\kappa(Y))$).

Note that 
for any $\tau$ and $\tau'$ in $s(\oper{Gal}(K^{\oper{sep}}/K))$ and $g\in G$, the equality $g\tau=\tau'g$ implies that the 
restrictions of $\tau$ and $\tau'$ to $K^{\oper{sep}}$ are equal, and therefore (since $s$ is a section) that $\tau=\tau'$. 
Therefore the normal core of 
$s(\oper{Gal}(K^{\oper{sep}}/K))$ in $\oper{Gal}(\kappa(\bar X)/\kappa(Y))$ can be described as:
$$\{\tau\in 
s(\oper{Gal}(K^{\oper{sep}}/K))|\forall \sigma\in \oper{Gal}(\kappa(\bar X)/\kappa(Y))\, \exists \tau'\in 
s(\oper{Gal}(K^{\oper{sep}}/K))\, \oper{s.t.} \sigma\tau=\tau'\sigma\}$$ $$=\{\tau\in s(\oper{Gal}(K^{\oper{sep}}/K))|\forall g\in G\, 
\exists \tau'\in s(\oper{Gal}(K^{\oper{sep}}/K))\, \oper{s.t.} g\tau=\tau'g\}$$ $$=\{\tau\in s(\oper{Gal}(K^{\oper{sep}}/K))|\forall g\in 
G\, 
 \oper{s.t.} g\tau=\tau g\}=V$$
\end{proof}

\begin{rmk}\label{E:explanation}
\rm
The situation of Hypotheses \ref{E:weaaak} in the special case that $K$ is a number field, and $Y=\p_K$, is of particular 
interest in the study of the Inverse Galois Problem. In this situation, recall from \cite{ch} that the field of 
moduli $M$ of $\bar X\rightarrow \p_{\bar \q}$ as a $G$-Galois branched cover is the
intersection of all of its fields of definition as a $G$-Galois branched cover, but is not
necessarily one itself. Moreover, since $M$ contains the field of moduli
of $\bar X \rightarrow \p_{\bar \q}$ as a mere cover, it is a field of definition as a mere cover. (The arguments in \cite{ch} 
in fact generalize easily to 
Hypotheses 
\ref{E:weaaak} with the extra assumptions that $K$ is a number field and that $Y$ has an unramified $K$-rational 
point. However, for 
simplicity's sake, we will continue to assume $Y=\p_K$ throughout the remainder of Part I.)
In light of the fact that every mere cover model has a unique minimal field of definition for its Galois action (as proven in Theorem 
\ref{E:tech}), one can explain the failure of $M$ to be a field of definition as a $G$-Galois branched cover as the
combination of two factors:
\begin{enumerate}
 \item Theorem \ref{E:tech} gives a unique minimal field of definition as a $G$-Galois branched cover for any particular mere cover model over $M$. However
each model might give a different minimal field of definition. Therefore the non-uniqueness of a model of $\bar X \rightarrow \p_{\bar
\q}$ over
$M$ contributes to the plurality of the minimal fields of definition.
\item If $L$ is an overfield of $M$, then there may be a mere cover model of $\bar X\rightarrow \p_{\bar \q}$ over
$L$ that does not descend to a mere cover model over $M$. Indeed, this is always the case if $G$ is not abelian, as the following 
construction shows.

Let $W$ be the compositum of all of the Galois field extensions of $M$ having a Galois group that is isomorphic to a 
subgroup of $\oper{Aut}(G)$. The field $W$ is clearly Galois over $M$, and is not equal to $\bar \q$. Let $L$ be a finite field extension 
of $W$. By Weissauer's Theorem (\cite{weissauer}), the field $L$ is Hilbertian. By Theorem \ref{E:tech}, every mere cover model over $M$ 
becomes Galois when base changed to $L$. Therefore there exists a $G$-Galois branched cover 
$X\rightarrow \p_{L}$ over $L$. In particular (since $L$ is Hilbertian), there exists an epimorphism $\epsilon: 
\oper{Gal}(L)\twoheadrightarrow G$ given by specializing.

Let $s:\oper{Gal}(L)\rightarrow \oper{Gal}(\kappa(\bar X)/L(x))$ be the section 
of
$$1\rightarrow G\rightarrow \oper{Gal}(\kappa(\bar X)/L(x))\rightarrow \oper{Gal}(L)\rightarrow 1$$
given by $\Psi^{-1}(\rho(\kappa(X)))$ for some
embedding $\rho$ of $\kappa(X)$ into $\kappa(\bar X)$. Let 
$s':\oper{Gal}(L)\rightarrow \oper{Gal}(\kappa(\bar X)/L(x))$ be the map defined by $\sigma\mapsto \epsilon(\sigma)s(\sigma)$. Since 
$s(\oper{Gal}(L))$ commutes with $G$ (by Lemma 
\ref{E:modelssections}), the map $s'$ is in fact a homomorphism. In fact, it is easy to check that $s'$ is a section of the above short 
exact sequence.

Since we assume that $G$ is not abelian, there exist elements $g$ and $h$ in $G$ that don't commute. Let $\tau$ be $\epsilon^{-1}(h)$. The 
elements $s'(\tau)$ and $g$ do not commute:
$$[s'(\tau),g]=[\epsilon(\tau)s(\tau),g]=[h\cdot s(\tau),g]=[h,g]\neq 1$$
Therefore, by Lemma \ref{E:modelssections}, the mere cover model of $\bar X\rightarrow \p_{\bar \q}$ over $L$ that 
corresponds to the section $s'$ is not Galois. Assume that this mere cover model descends to $M$. Then by Theorem \ref{E:tech}, it must 
be Galois when base changed to $L$, in contradiction to what we have shown above. Therefore we have constructed a mere cover model of 
$\bar X\rightarrow \p_{\bar \q}$ over $L$ that does not descend to $M$.
\end{enumerate}
\end{rmk}

As an immediate corollary of Theorem \ref{E:tech} we have:
\begin{crl}\label{E:thebigone} Let $K$ be a number field, let $G$ be a finite group, and let $\bar X\rightarrow \p_{\bar \q}$ be a 
$G$-Galois 
branched cover over $\bar \q$. Let $F$ be the field of
moduli of
$\bar X\rightarrow \p_{\bar \q}$ as a mere cover, and let $M$ be the field of moduli of $\bar X\rightarrow \p_{\bar \q}$ as a
$G$-Galois branched cover. Then $M$ is Galois over $F$ with Galois group a subquotient of $\oper{Aut}(G)$.
\end{crl}
\begin{proof}
In light of Theorem \ref{E:tech}, it suffices
to show that $M$ is Galois over $F$. As mention in Remark \ref{E:explanation}, it was proven in \cite{ch} 
that $M$ is the intersection of all of the fields of definition as a $G$-Galois branched cover.
It therefore suffices to prove that for every field of definition $L$ of $\bar X\rightarrow \p_{\bar \q}$ as a $G$-Galois branched
cover, and for every $\sigma$ in $\oper{Gal}(\bar \q/F)$, the field $\sigma L$ is also a field of definition as a $G$-Galois branched cover.
Let $X\rightarrow \p_L$ be an $L$-model as a $G$-Galois branched cover, and let $X\times_L\sigma L\rightarrow \p_{\sigma L}$ be its twist by
$\sigma$. This cover is clearly Galois. Furthermore, note that $X\times_L\sigma L\rightarrow \p_{\sigma L}$ is a mere cover model over $\sigma
L$ of the cover
$\bar X \rightarrow \p_{\bar \q}$ after it has been twisted by $\sigma$. Indeed, by the definition of $F$, the cover resulting from twisting
$\bar X\rightarrow \p_{\bar \q}$ by $\sigma$ is isomorphic to $\bar X\rightarrow \p_{\bar \q}$ as a mere cover. Therefore $\sigma
L$ is a field of definition of $\bar X\rightarrow \p_{\bar \q}$ as a mere cover, and $X\times_L\sigma L\rightarrow \p_{\sigma L}$ is a mere
cover model of this cover that is Galois. In other words, the field $\sigma L$ is field of definition of $\bar X\rightarrow \p_{\bar
\q}$ as a $G$-Galois branched cover, which is what we wanted to prove.
\end{proof}
\section{Part II - $K$ is a Field of Definition as a $G$-Galois Branched Cover}
This section will refer to the following situation.
\begin{hyp}\label{E:hyp1}\rm
The same as Hypotheses \ref{E:superweaaak}, with the extra assumption that $K$ is a field of definition of $\bar X \rightarrow 
\bar Y$ as a $G$-Galois branched cover.
\end{hyp}

Whenever we assume the above hypotheses, we will use the following notation.
\begin{ntt}\label{E:not1}\rm
In the situation of Hypotheses \ref{E:hyp1},  let $\bar B$ be the branch divisor of $\bar X \rightarrow 
\bar Y$, and let 
$B \subset Y$ 
be the Zariski closure of $\bar B$ in $Y$. Let $Y^*=Y\smallsetminus B$ (and $\bar Y^*=Y^*_{K^{\oper{sep}}}$), and fix a geometric 
point $\underline t_0$ of $\bar Y^*$.  We have the short exact sequence of \'etale fundamental groups:
$$1\rightarrow \pi_1(\bar Y^*,\underline t_0)\rightarrow \pi_1(Y^*,\underline 
t_0)\rightarrow \oper{Gal}(K^{\oper{sep}}/K)\rightarrow 1$$
Every  
$K$-rational 
point $P$ of $Y^*$ induces a section, which is defined up to conjugation by an element of 
$\pi_1(Y^*,\underline t_0)$. We denote this associated section by the notation $s_P$.

Every $G$-Galois model $X\rightarrow Y$ defines an epimorphism $\beta^{X}:\pi_1(Y^*,\underline t_0)\twoheadrightarrow G$ up 
to conjugation by an element of $G$. For every $P\in Y^*(K)$, let 
$\phi^{X}_P=\beta^{X}\circ s_P$. 
Notice that it too is well defined up to conjugation by an element of $G$.
\end{ntt}
\begin{dff}\label{E:defoflift}\rm
If a homomorphism $\varphi$ from $\oper{Gal}(K^{\oper{sep}}/K)$ to $G$ is equal to $\phi_P^{X}$ when considered modulo conjugation in $G$, 
then we say 
that 
$\varphi$ is a \it lift \rm of $\phi_P^{X}$.
\end{dff}
\begin{rmk}\label{E:rmkofspec}\rm
Note that for every $P\in Y^*(K)$ the field $(K^{\oper{sep}})^{\ker (\phi_P^{X})}$ is well defined (i.e., is independent of the lift of 
$\phi_P^{X}$), and that the specialization 
of $X\rightarrow Y$ at $P$ is an \'etale algebra 
extension 
$\oper{Spec}(\prod_{i=1}^mL_i)\rightarrow \oper{Spec}(K)$, where the $L_i$'s are all isomorphic to $(K^{\oper{sep}})^{\ker (\phi_P^{X})}$ as mere field extensions over $K$. This observation justifies calling $(K^{\oper{sep}})^{\ker (\phi_P^{X})}$ the \it field induced by specializing 
$X\rightarrow 
Y$ at $P$\rm. 
\end{rmk}
\begin{ntt}\rm
 In the situation of Hypotheses \ref{E:hyp1}, let $X\rightarrow Y$ be a $G$-Galois model of $\bar X\rightarrow 
\bar Y$, and 
let $P$ be in $Y^*(K)$. We will denote the field $(K^{\oper{sep}})^{\ker (\phi_P^{X})}$ induced by specializing $X\rightarrow Y$ at $P$ by 
$L_P^{X}$. 
Furthermore, we will use the notation
$$L_P=\cap_{\{X/Y\mbox{ a $G$-Galois model}\}} 
L_P^{X}.$$
\end{ntt}
\begin{cnv}\rm
 In the remainder of this section, I will use the convention that, in the situation of Hypotheses 
\ref{E:hyp1}, mere cover models of $\bar X\rightarrow \bar Y$ that are not assumed to be Galois will be adorned 
with a tilde 
(e.g., $\tilde X \rightarrow Y$), whereas ones that are assumed to be Galois will be written without a tilde (e.g., $X\rightarrow Y$).
\end{cnv}
The main theorem of this section says that in the situation of Hypotheses \ref{E:hyp1}, as opposed to the situation of Hypotheses \ref{E:weaaak},
 we are able to give an explicit field that contains the minimal field of Galois action of any mere cover model $X\rightarrow Y$, assuming $X$ has a
 $K$-rational point.
\begin{thm}\label{E:crux}
 In the situation of Hypotheses $\ref{E:hyp1}$,
let $\tilde X\rightarrow Y$ be a mere cover model of $\bar X \rightarrow \bar Y$ that has an unramified 
$K$-rational point above some point $P$ in $Y^*(K)$. Then $\tilde X\rightarrow Y$ becomes Galois when base changed to $L_P$. 
Furthermore, if $X\rightarrow Y$ is a 
particular $G$-Galois 
model, then the base change of $\tilde X\rightarrow Y$ to $L_P^{X}$ is isomorphic to the base 
change of $X\rightarrow Y$ to $L_P^{X}$ as mere cover models.
\end{thm}
The same proof mechanism that gives Theorem \ref{E:crux} also gives the following theorem for free:
\begin{thm}\label{E:specialization}
In the situation of Hypotheses \ref{E:hyp1}, there exists an equivalence relation ``$\equiv$'' on $Y^*(K)$ such that the following 
hold:
\begin{enumerate}
\item 
For every mere cover model $\tilde X \rightarrow Y$ of $\bar X\rightarrow \bar Y$, if $\tilde X$ has an 
unramified 
$K$-rational point then there exists a $P\in Y^*(K)$, and a natural number $d$, which is 
divisible by $|Z(G)|$ and divides $|G|$, such that for every point $Q\in Y^*(K)$ the fiber over $Q$ contains exactly $d$ many 
$K$-rational 
points if $Q\equiv P$ and $0$ otherwise.
  \item For every $P\in Y^*(K)$ there exists a unique (up to isomorphism) mere cover model $\tilde X\rightarrow Y$ of 
$\bar X\rightarrow \bar Y$ such that the unramified $K$-rational points 
of $\tilde X$ lie exactly above its equivalence class $[P]_{\equiv}$.
Furthermore, if $d$ is the number of $K$-rational points of $\tilde 
X$ in the fiber 
of each $Q\in [P]_{\equiv}$,  then $\tilde X\rightarrow Y$ has precisely $|G|/d$ many Galois conjugates.
\item Let $X\rightarrow Y$ be a $G$-Galois model of $\bar X \rightarrow \bar Y$ over $K$, and let $P$ and $Q$ be 
two points in 
$Y^*(K)$. Then $P\equiv Q$ if and only if $\phi_P^{X}$ and $\phi_Q^{X}$ are equal (up to conjugation in $G$). 
\end{enumerate}
\end{thm}

\begin{rmk}\rm \label{E:remarkaftermain}
Notice that the Galois group $\oper{Gal}(L_P^{X}/K)$ comes equipped with a homomorphism to $G$ (unique up to 
conjugation 
in $G$). In particular if 
$\oper{img}(\phi_P^{X})=G$  then $L_P^{X}/K$ is equipped with the structure of a $G$-Galois extension of $K$. 
Furthermore, if $P$ and $Q$ are two $K$-rational points of $Y$ such that
$\oper{img}(\phi^{X}_{P})=\oper{img}(\phi^{X}_{Q})=G$ (or equivalently if the specializations of 
$X/Y$ at $P_1$ and $P_2$ are $G$-Galois field extensions of $K$) then $\phi_{P}^{X}$ and $\phi_{Q}^{X}$ are equal (modulo 
conjugation in $G$) if and only if the field extensions induced by specializing $X\rightarrow Y$ at $P$ and at $Q$ are 
\it isomorphic as $G$-Galois extensions of $K$. \rm
\end{rmk}

Both Theorem \ref{E:crux} and Theorem \ref{E:specialization} follow from the observation (Lemma \ref{E:twisted}) that every mere cover model of $\bar X\rightarrow \bar Y$
over $K$ is the ``twist'' of every $G$-Galois model, in the sense that I will describe below. (It is in this context that Theorem \ref{E:specialization} can be viewed as a
reverse version of ``The Twisting Lemma''; see \cite{debesgrunwald}, as well as Lemma \ref{E:twistinglemma} appearing in the following subsection.)

\subsection{Twisted Covers}
We will be using the notion of twisting a $G$-Galois model by a homomorphism from $\oper{Gal}(K^{\oper{sep}}/K)$ to $G$. This concept can be viewed
as a special case of the notion of a ``contracted product'' in the theory of torsors. (See \cite{dem}, III.4.1; \cite{torsy}, 2.2.) Pierre D\`ebes, in a series of papers
beginning with \cite{twist}, has applied this notion in order
to study the specializations of $G$-Galois branched covers.

So that I may define ``twists'' I will need the following well-known proposition:

\begin{prp}\label{E:correspondence}
 Let $Y$ be an integral regular scheme, and let $B$ be a reduced codimension $1$ subscheme of $Y$. Let $\underline t_0$ be a geometric 
point of $Y\smallsetminus B$. Then there is a bijection between the following:
 \begin{enumerate}
  \item Homomorphisms from 
$\pi_1(Y\smallsetminus B, \underline t_0)$ to $S_n$, modulo conjugation in $S_n$.
  \item Equivalence classes of (mere) covers of $Y$ of degree $n$ that are unramified away from $B$.
  \end{enumerate}
  This bijection is given by mapping $\gamma:\pi_1(Y\smallsetminus B, \underline t_0)\rightarrow S_n$ to the cover defined by the subgroup 
$\gamma^{-1}(\{\sigma\in S_n|\sigma(1)=1\})$ of $\pi_1(Y\smallsetminus B, \underline t_0)$.
\end{prp}
\begin{dff}\label{E:defoftwist}\rm
 In the situation of Hypotheses \ref{E:hyp1}, let 
$r:\pi_1(Y^*,\underline t_0)\rightarrow \oper{Gal}(K^{\oper{sep}}/K)$ be the map induced by the structure morphism, let $S_G$ denote 
the group of permutation on the elements of the group $G$, and let $X\rightarrow Y$ be the $G$-Galois model of 
$\bar X\rightarrow 
\bar Y$ over $K$ induced by 
an epimorphism 
$\beta:\pi_1(Y^*,\underline t_0)\twoheadrightarrow G$. Let $\alpha$ be some homomorphism from $\oper{Gal}(K^{\oper{sep}}/K)$ to $G$. 
We 
define 
\it the twist $X^{\alpha}\rightarrow Y$ of $X\rightarrow Y$ by $\alpha$ \rm to be the mere cover model induced (as in 
Proposition \ref{E:correspondence}) by 
the homomorphism:
$$\pi_1(Y^*,\underline t_0)\rightarrow S_{G}$$
given by $\sigma\mapsto f_{\sigma}$, where $f_{\sigma}$ is the permutation on the elements of $G$ given by \\$h\mapsto 
\beta(\sigma)\cdot h\cdot (\alpha(r(\sigma)))^{-1}$.
\end{dff}
\begin{rmk}\rm
\begin{enumerate}\label{E:credits}\
 \item
Note that since $\pi_1(\bar Y^*, \underline t_0)$ is equal to the kernel of $r$, the twisted cover 
$X^{\alpha}\rightarrow Y$ is in fact a mere cover model of $\bar X \rightarrow \bar Y$.

 Furthermore, an easy check shows that Definition \ref{E:defoftwist} depends only on the $G$-Galois model, and not on the 
epimorphism $\beta$. This follows from the fact that the $G$-Galois model $X\rightarrow Y$ fixes $\beta$ up to conjugation in $G$. 
Indeed, if $\beta'=g^{-1}\beta g$ for some $g\in G$, then the homomorphism from $\pi_1(Y^*,\underline t_0)$ to $S_{G}$ 
associated to $\beta'$ is conjugate to the one associated to $\beta$ by the permutation that takes $h$ to $g^{-1}h$. 
\item
I have chosen to loosely follow the conventions used in \cite{debesgrunwald}. Texts that use the ``contracted product'' definition
(as appearing in \cite{torsy}, 2.2) tend not to focus on the case where the structure group is a finite constant group scheme, and in
particular don't highlight the study of mere cover models of a given $G$-Galois branched cover over $K^{\oper{sep}}$.
\end{enumerate}
\end{rmk}
The following is a slight strengthening (see Remark \ref{E:howslight}) of the version of the ``Twisting Lemma'' appearing in 
\cite{debesgrunwald}, Section 2.
\begin{lem} \label{E:twistinglemma} (The Twisting Lemma)
 In the situation of Hypotheses \ref{E:hyp1}, let $P$ be an unramified $K$-rational point of $Y$ and let 
$\alpha:\oper{Gal}(K^{\oper{sep}}/K)\rightarrow G$ be a homomorphism. Then:
\begin{enumerate}
 \item The homomorphism $\alpha$ is a lift of $\phi_P^{X}$ (see Definition \ref{E:defoflift}) if and only if $X^{\alpha}\rightarrow 
Y$ has a $K$-rational point over $P$.
\item For such a point, the number of $K$-rational 
points above $P$ is equal to the order of the centralizer $C_G(\oper{img}(\alpha))$ of $\oper{img}(\alpha)$ in $G$.
\end{enumerate}
\end{lem}

\begin{proof}

 The point $P$ induces a section $s_P:\oper{Gal}(K^{\oper{sep}}/K)\rightarrow \pi_1(Y^*, \underline t_0)$, defined up to 
conjugation in $\pi_1(Y^*, \underline t_0)$. For the remainder of the proof, fix a representative of $s_P$, and fix an 
epimorphism $\beta:\pi_1(Y^*,\underline t_0)\twoheadrightarrow G$ that induces the $G$-Galois model $X\rightarrow Y$.
Note that $\varphi_P^{X}:=\beta \circ s_P$ is a lift of $\phi_P^{X}$.

The set of 
rational points above $P$ is in bijection with the set of elements of $G$ fixed by $\{f_{s_P(\tau)}|\tau\in 
\oper{Gal}(K^{\oper{sep}}/K)\}\subset S_G$. 
(See Definition \ref{E:defoftwist}.) In particular there exists a $K$-rational point in $X^{\alpha}$ above $P$ if and only if there 
exists an element $h\in G$ such that for every $\tau\in \oper{Gal}(K^{\oper{sep}}/K)$ we have 
$h=\beta(s_P(\tau))\cdot h\cdot \alpha(r(s_P(\tau)))^{-1}(=\varphi_P^{X}(\tau)\cdot h\cdot \alpha(\tau)^{-1})$. Therefore, there exists a 
$K$-rational point in $X^{\alpha}$ above $P$ if and only if there exists an $h\in G$ such that $\varphi_P^{X}=h\alpha 
h^{-1}:\oper{Gal}(K^{\oper{sep}}/K)\rightarrow G$, proving the first assertion in the lemma.

Assume there exists an $h\in G$ as above. 
Then 
$h'\in G$ also satisfies $\varphi_P^{X}=h' \alpha {h'}^{-1}$ if and only if 
$\alpha=h'^{-1}h\alpha(h'^{-1}h)^{-1}$. In 
other 
words, if and only if ${h'}^{-1}$  is an element of $C_G(\oper{img}(\alpha))h^{-1}$. Therefore, if there exists 
at least one $K$-rational point in 
$X^{\alpha}$ above $P$, there exist precisely $|C_G(\oper{img}(\alpha))|$ many.
\end{proof}
\begin{rmk}\label{E:howslight}\rm
 Assertion (1) of Lemma \ref{E:twistinglemma} above is precisely Lemma 2.1 in \cite{debesgrunwald}, and assertion (2) is the strengthening. 
I will 
need assertion (2) in the proof of Theorem \ref{E:specialization}.
\end{rmk}
The remainder of this subsection is devoted to proving a few observations about twisted covers that will be helpful
in the proof of Theorems \ref{E:crux} and \ref{E:specialization}. We will freely make use of the set $\Omega$ and bijection $\Psi$
introduced in Part I of this paper. 
\begin{ntt}\rm
 In order to simply notation, for every field $F$ in $\Omega$ we will denote the section $\Psi^{-1}(F)$ by $w_F$.
\end{ntt}
Under Hypotheses \ref{E:hyp1}, we have the following short exact sequence:
$$1\rightarrow \oper{Gal}(\kappa(\bar X)/\kappa(\bar Y))\rightarrow 
\oper{Gal}(\kappa(\bar X)/\kappa(Y))\xrightarrow[]{r'} \oper{Gal}(K^{\oper{sep}}/K)\rightarrow 1$$

Every field $F$ in $\Omega$ induces an isomorphism of group 
extensions:
$$\begindc{\commdiag}[330]
\obj(-1,0)[aa]{$1$}
\obj(2,0)[ab]{$\oper{Gal}(\kappa(\bar X)/\kappa(\bar Y))$}
\obj(6,0)[ac]{$\oper{Gal}(\kappa(\bar X)/\kappa(Y))$}
\obj(9,0)[ad]{$\oper{Gal}(K^{\oper{sep}}/K)$}
\obj(11,0)[ae]{$1$}
\mor{aa}{ab}{}
\mor{ab}{ac}{}
\mor{ac}{ad}{$r'$}
\mor{ad}{ae}{}
\obj(-1,2)[ba]{$1$}
\obj(2,2)[bb]{$G$}
\obj(6,2)[bc]{$G\times \oper{Gal}(K^{\oper{sep}}/K)$}
\obj(9,2)[bd]{$\oper{Gal}(K^{\oper{sep}}/K)$}
\obj(11,2)[be]{$1$}
\mor{ba}{bb}{}
\mor{bb}{bc}{$i$}
\mor{bc}{bd}{$p$}
\mor{bd}{be}{}
\mor{bb}{ab}{$\Phi$}
\mor{bc}{ac}{$\Phi\times w_F$}
\mor{bd}{ad}{}
\enddc$$
with inverse $(\Phi\times w_F)^{-1}:\oper{Gal}(\kappa(\bar X)/\kappa(Y))\rightarrow G\times 
\oper{Gal}(K^{\oper{sep}}/K)$ taking  $\sigma$ 
to 
$(\Phi^{-1}(\sigma|_F),r'(\sigma))$.

There is a natural bijection between $\oper{Hom}(\oper{Gal}(K^{\oper{sep}}/K),G)$ and $\oper{Sec}(p)$ taking a homomorphism 
$\alpha\in\oper{Hom}(\oper{Gal}(K^{\oper{sep}}/K),G)$ to the section that takes $\tau \in \oper{Gal}(K^{\oper{sep}}/K)$ to 
$(\alpha(\tau),\tau)$. Therefore, via this 
isomorphism of group extensions above, the field $F$ induces a bijection\\ 
$\Sigma_F:\oper{Hom}(\oper{Gal}(K^{\oper{sep}}/K),G)\rightarrow \oper{Sec}(r')$ taking a homomorphism 
$\alpha\in\oper{Hom}(\oper{Gal}(K^{\oper{sep}}/K),G)$ to the section that takes $\tau\in \oper{Gal}(K^{\oper{sep}}/K)$ to 
$\Phi(\alpha(\tau))w_F(\tau)$.

\begin{lem}\label{E:twisted}
In the situation above, let $X\rightarrow Y$ be a $G$-Galois model of $\bar X\rightarrow \bar Y$, and let 
$\alpha$ be a 
homomorphism in $\oper{Hom}(\oper{Gal}(K^{\oper{sep}}/K), G)$. Let $\rho$ be a $\kappa(Y)$-embedding of $\kappa(X)$ into 
$\kappa(\bar X)$. Then the mere cover model $X^{\alpha}\rightarrow Y$ is isomorphic
to the mere cover model associated with $\Psi(\Sigma_{\rho(\kappa(X))}(\alpha))$. In particular, since $\Psi\circ \Sigma_{\rho(\kappa(X))}$
is a bijection, every mere cover model is the twist of $X\rightarrow Y$.
\end{lem}
\begin{proof}
Let $\underline t_1$ be a geometric point of $\bar X$ lying above $\underline t_0$, and let $R$ be the ramification divisor of 
$X\rightarrow Y$. Then we have the following commutative diagram:
$$\begindc{\commdiag}[330]

\obj(-1,0)[aa]{$1$}
\obj(2,0)[ab]{$G\cong\oper{Gal}(\kappa(\bar X)/\kappa(\bar Y))$}
\obj(6,0)[ac]{$\oper{Gal}(\kappa(\bar X)/\kappa(Y))$}
\obj(9,0)[ad]{$\oper{Gal}(K^{\oper{sep}}/K)$}
\obj(11,0)[ae]{$1$}
\mor{aa}{ab}{}
\mor{ab}{ac}{$i'$}
\mor{ac}{ad}{$r'$}
\mor{ad}{ae}{}
\obj(-1,1)[ba]{$1$}
\obj(2,1)[bb]{$\pi_1(\bar Y\smallsetminus \bar B,\underline t_0)$}
\obj(6,1)[bc]{$\pi_1(Y\smallsetminus B,\underline t_0)$}
\obj(9,1)[bd]{$\oper{Gal}(K^{\oper{sep}}/K)$}
\obj(11,1)[be]{$1$}
\mor{ba}{bb}{}
\mor{bb}{bc}{$i$}
\mor{bc}{bd}{$r$}
\mor{bd}{be}{}
\obj(2,2)[cb]{$\pi_1(\bar X\smallsetminus \bar R,\underline t_1)$}
\obj(6,2)[cc]{$\pi_1(\bar X\smallsetminus \bar R,\underline t_1)$}
\obj(2,3)[db]{$1$}
\obj(6,3)[dc]{$1$}
\obj(2,-1)[eb]{$1$}
\obj(6,-1)[ec]{$1$}
\mor{db}{cb}{}
\mor{cb}{bb}{$f$}
\mor{bb}{ab}{$g$}
\mor{ab}{eb}{}
\mor{dc}{cc}{}
\mor{cc}{bc}{$f'$}
\mor{bc}{ac}{$g'$}
\mor{ac}{ec}{}
\mor{cb}{cc}{}[\atright,\equalline]
\mor{ad}{bd}{}[\atright,\equalline]
\enddc$$

Let $\beta':\oper{Gal}(\kappa(\bar X)/\kappa(Y))\twoheadrightarrow G$ be defined by taking $\sigma$ to 
$\Phi^{-1}(\sigma|_{\rho(\kappa(X))})$, and let $\beta:\pi_1(Y\smallsetminus B,\underline t_0)\twoheadrightarrow G$ be the composition 
$\beta'\circ g'$.

By Definition \ref{E:defoftwist} and Proposition \ref{E:correspondence}, it 
follows that the 
mere cover model $X^{\alpha}\rightarrow Y$ corresponds to the subgroup of $\pi_1(Y\smallsetminus B,\underline{t}_0)$ defined by 
$$H_{\alpha}=\{\sigma \in \pi_1(Y\smallsetminus 
B,\underline{t}_0)|\alpha(r(\sigma))=\beta(\sigma)\}.$$

In other words, there exists a $\kappa(Y)$-embedding
$\rho^{\alpha}$ of $\kappa(X^{\alpha})$ into $\kappa(\bar X)$ such that \\$\oper{Gal}(\kappa(\bar X)/\rho^{\alpha}(\kappa(X^{\alpha})))$ is equal to $H_{\alpha}/(\pi_1(\bar X\smallsetminus \bar R,\underline t_1))$.
(The group $H_{\alpha}$ contains $\pi_1(\bar X\smallsetminus 
\bar R,\underline t_1)$ because $X^{\alpha}\rightarrow Y$ is a mere cover model.)

By the commutativity of the above diagram we have the 
presentation:$$H_{\alpha}/(\pi_1(\bar X\smallsetminus \bar R,\underline t_1))=\{\sigma\in 
\oper{Gal}(\kappa(\bar X)/\kappa(Y))| 
\alpha(r'(\sigma))=\beta'(\sigma)\}$$

Let $p_1$ and $p_2$ be the projection maps of $G\times \oper{Gal}(K^{\oper{sep}}/K)$ to $G$ and $\oper{Gal}(K^{\oper{sep}}/K)$ respectively. 
It is easy to see that 
$p_1\circ(\Phi\times w_{\rho(\kappa(X))})^{-1}=\beta'$ and $p_2\circ(\Phi\times w_{\rho(\kappa(X))})^{-1} =r'$. Therefore:
$$H_{\alpha}/(\pi_1(\bar X\smallsetminus \bar R,\underline t_1))=\{\sigma\in 
\oper{Gal}(\kappa(\bar X)/\kappa(Y))| 
\alpha(r'(\sigma))=\beta'(\sigma)\}$$ $$=\{\sigma\in \oper{Gal}(\kappa(\bar X)/\kappa(Y))| 
\alpha((p_2\circ(\Phi\times w_{\rho(\kappa(X))})^{-1})(\sigma))=(p_1\circ(\Phi\times 
w_{\rho(\kappa(X))})^{-1})(\sigma)\}$$ $$=(\Phi\times w_{\rho(\kappa(X))})(\{(h,\tau)\in G\times 
\oper{Gal}(K^{\oper{sep}}/K)|h=\alpha(\tau)\})=\oper{img}(\Sigma_{\rho(\kappa(X))}(\alpha))$$

Therefore $\rho^{\alpha}(\kappa(X^{\alpha}))=\Psi(\Sigma_{\rho(\kappa(X))}(\alpha))$, which implies that $X^{\alpha}\rightarrow Y$ is isomorphic to the mere
cover model induced by $\Psi(\Sigma_{\rho(\kappa(X))}(\alpha))$.

\end{proof}
\begin{rmk}\label{E:someall}
 \rm
  In the situation of Lemma \ref{E:twisted}, note that different choices of $G$-Galois models would yield different 
bijections between 
$\oper{Hom}(\oper{Gal}(K^{\oper{sep}}/K),G)$ and the set $\Omega$. This is because different 
$G$-Galois models 
yield different isomorphisms of $\oper{Gal}(\kappa(\bar X)/\kappa(Y))$ with $G\times \oper{Gal}(K^{\oper{sep}}/K)$.

Furthermore, as an immediate consequence of Lemma \ref{E:twisted} we see that every mere cover model is not only the twist of \it some 
\rm $G$-Galois model, but of \it every \rm $G$-Galois model.
\end{rmk}
\begin{lem}\label{E:galcenter}
 In the situation of Hypotheses \ref{E:hyp1}, let $X\rightarrow Y$ be a $G$-Galois model of $\bar X\rightarrow 
\bar Y$, and let 
$\alpha$ be a homomorphism $\oper{Gal}(K^{\oper{sep}}/K)\rightarrow G$. Then the twisted cover 
$X^{\alpha}\rightarrow Y$ is Galois if and only if the image of $\alpha$ is contained in $Z(G)$.
\end{lem}
\begin{proof}
 Let $\rho$ be a $\kappa(Y)$-embedding of $\kappa(X)$ into $\kappa(\bar X)$. By Lemma \ref{E:twisted}, 
 there exists a $\kappa(Y)$-embedding $\rho^{\alpha}$ of $\kappa(X^{\alpha})$ into $\kappa(\bar X)$ such that $\rho^{\alpha}(\kappa(X^{\alpha}))=\Psi(\Sigma_{\rho(\kappa(X))}(\alpha))$.
 
 By Lemma \ref{E:modelssections}, the mere cover model $X^{\alpha}\rightarrow Y$ is Galois if and only if the image of $\Sigma_{\rho(\kappa(X))}(\alpha)$
 commutes with $G$. Thus, the mere cover $X^{\alpha}\rightarrow Y$ is 
Galois if and only if for every $\tau\in\oper{Gal}(K^{\oper{sep}}/K)$ the element $(\alpha(\tau),\tau)\in 
G\times\oper{Gal}(K^{\oper{sep}}/K)$ commutes with  
$G\times 1$. Equivalently, if and only if the image of $\alpha$ is contained in $Z(G)$.
\end{proof}
\subsection{Proof of Theorems \ref{E:crux} and \ref{E:specialization}}
\begin{dff}\rm
In the situation of Hypotheses \ref{E:hyp1}, for every $G$-Galois model $X\rightarrow Y$  of $\bar X\rightarrow 
\bar Y$ 
over $K$, define an equivalence relation 
``$\equiv_{X/Y}$'' on $Y^*(K)$ by: $$P\equiv_{X/Y} Q \iff \phi^{X}_{P} \mbox{ and } \phi^{X}_{Q} \mbox{ are equal (up to 
conjugation by an element of $G$)}.$$
\end{dff}
\begin{lem}\label{E:yay}
 In the situation of Hypotheses \ref{E:hyp1}, let $X\rightarrow Y$ and $X'\rightarrow Y$ be two $G$-Galois models of 
$\bar X \rightarrow \bar Y$. Then the equivalence relations ``$\equiv_{X/Y}$'' and ``$\equiv_{X'/Y}$'' on 
$Y^*(K)$ are equal.
\end{lem}
\begin{proof}
 Let $P$ and $Q$ be two 
points in $Y^*(K)$, and let $\varphi_P^{X}$ be a lift of $\phi_P^{X}$.  By the Twisting Lemma (Lemma \ref{E:twistinglemma}), the 
condition $P\equiv_{X/Y} Q$ is equivalent 
to the existence of a $K$-rational point in $X^{\varphi_P^{X}}$ above $Q$. We will show that $P\equiv_{X'/Y}Q$ is also equivalent 
to this condition.

By Lemma 
\ref{E:twisted} and Remark \ref{E:someall}, the mere cover model $X^{\varphi_P^{X}}\rightarrow Y$ is a twist not only of $X\rightarrow Y$, but also 
of 
$X'\rightarrow Y$. In other words, there exists a homomorphism $\alpha:\oper{Gal}(K^{\oper{sep}}/K)\rightarrow G$, such that 
$X^{\varphi_P^{X}}\rightarrow Y$ is isomorphic to ${X'}^{\alpha}\rightarrow Y$ as mere covers. By Lemma 
\ref{E:twistinglemma}, we see that $X^{\varphi_P^{X}}\rightarrow Y$ (and therefore also ${X'}^{\alpha}\rightarrow Y$) has a 
$K$-rational point above $P$. Therefore, again by Lemma \ref{E:twistinglemma}, we see that $\alpha$ is a lift of $\phi^{X'}_P$. 
Therefore the condition $P\equiv_{X'/Y} Q$ is equivalent to $\alpha$ being a lift of $\phi_Q^{X'}$, which by Lemma 
\ref{E:twistinglemma} is equivalent to the existence of a $K$-rational point in ${X'}^{\alpha}$ above $Q$. Since 
$X^{\varphi_P^{X}}\rightarrow Y$ and ${X'}^{\alpha}\rightarrow Y$ are isomorphic as mere cover models, we proved that 
the condition $P\equiv_{X'/Y} Q$ is equivalent to the existence of a $K$-rational point in $X^{\varphi_P^{X}}$ above $Q$, as we 
wanted to prove.

\end{proof}
Since we proved that $\equiv_{X/Y}$ is independent of the $G$-Galois model, we may denote it simply by $\equiv$. We will now show 
that this equivalence relation satisfies Theorem \ref{E:specialization}.
\begin{proof} (of Theorem \ref{E:specialization})

\noindent\bf Proof of (1): \rm Let $X\rightarrow Y$ be a $G$-Galois model of $\bar X\rightarrow \bar Y$, and 
let $\tilde X\rightarrow Y$ be a mere cover model of it. Let 
$P$ be a point in $Y^*(K)$ such that its fiber in $\tilde X$ has a rational point. By Lemma \ref{E:twisted}, there exists a 
homomorphism 
$\alpha:\oper{Gal}(K^{\oper{sep}}/K)\rightarrow G$ such that $\tilde X\rightarrow Y$ is isomorphic to $X^{\alpha}\rightarrow Y$ as 
mere covers. 
Since $X^{\alpha}$ has a $K$-rational point over $P$, then by Lemma \ref{E:twistinglemma} it follows that $\alpha$ is a 
lift of $\phi_P^{X}$. Therefore, by definition, a point $Q$ in $Y^*(K)$ satisfies $P\equiv Q$ if and 
only if $\alpha$ is a lift of $\phi_Q^{X}$. Again by Lemma \ref{E:twistinglemma}, this implies that 
$P\equiv Q$ if and only if $\tilde X$ has a $K$-rational point over $Q$. Furthermore, by assertion (2) of Lemma \ref{E:twistinglemma}, 
the 
number of $K$-rational point in $\tilde X$ above 
any $Q\in 
[P]_{\equiv}$ is equal to $|C_G(\oper{img}(\alpha))|$, and therefore divides $|G|$ and is divisible by $|Z(G)|$.

\noindent\bf Proof of (2): \rm Let $X\rightarrow Y$ be a 
$G$-Galois 
model of $\bar X \rightarrow \bar Y$, and let $\varphi_P^{X}$ be a lift of $\phi_P^{X}$. By Lemma 
\ref{E:twistinglemma}, the 
twisted cover $X^{\varphi_P^{X}}\rightarrow Y$ has a $K$-rational point above $P$. Therefore, by assertion (1) of Theorem 
\ref{E:specialization}, the twisted cover $X^{\varphi_P^{X}}\rightarrow Y$ satisfies that its $K$-rational points lie exactly above 
$[P]_{\equiv}$. This proves existence.

By Lemma \ref{E:twisted} all mere covers are isomorphic to twists of $X\rightarrow Y$ by homomorphisms from 
$\oper{Gal}(K^{\oper{sep}}/K)$ to 
$G$. For such a homomorphism $\alpha$, Lemma \ref{E:twistinglemma} implies that $X^{\alpha}$ has a $K$-rational point above $P$ (and 
therefore, by (1), exactly above the equivalence class $[P]_{\equiv}$) if and only if $\alpha$ is conjugate to $\varphi_P^{X}$ in $G$.
In other words, if and only if $\alpha$ and $\varphi_P^{X}$ induce isomorphic mere cover models.

The number of Galois conjugates of $X^{\varphi_P^{X}}\rightarrow Y$ is equal to the number of different homomorphisms from $\oper{Gal}(K^{\oper{sep}}/K)$ to $G$ that are conjugate to $\varphi_P^{X_K}$ in $G$,
namely to  $(G:C_G(\oper{img}(\varphi_P^{X_K})))$.  Therefore, since $d=|C_G(\oper{img}(\alpha))|$, we are done

\noindent\bf Proof of (3): \rm Follows from the fact that $\equiv$ is well defined (Lemma \ref{E:yay}).
\end{proof}
We can now prove Theorem \ref{E:crux}:
\begin{proof} (of Theorem \ref{E:crux})

Let $X\rightarrow Y$ be a $G$-Galois model of $\bar X\rightarrow \bar Y$. By Lemma \ref{E:twisted}, there 
exists a 
homomorphism 
$\alpha:\oper{Gal}(K^{\oper{sep}}/K)\rightarrow G$ such that $\tilde X\rightarrow Y$ is isomorphic to $X^{\alpha}\rightarrow Y$ as 
mere cover 
models. Furthermore, it 
is 
easy to see from Lemma \ref{E:twisted} that for every overfield $L/K$, the base change of the twisted cover $X^{\alpha}\rightarrow 
Y$ to $L$ is just the twist of $X\rightarrow Y$ by the restriction of $\alpha$ to $\oper{Gal}(K^{\oper{sep}}/L)$. Therefore, Lemma 
\ref{E:galcenter} implies that the field $(K^{\oper{sep}})^{\alpha^{-1}(Z(G))}$ is the unique minimal field where $X^{\alpha}\rightarrow 
Y$ 
becomes 
Galois.

Since $\tilde X$ has a $K$-rational point above $P$, by Lemma \ref{E:twistinglemma} the homomorphism $\alpha$ is a lift of 
$\phi_P^{X}$. 
Therefore $(K^{\oper{sep}})^{\alpha^{-1}(Z(G))}$ is contained in $(K^{\oper{sep}})^{\ker \phi_P^{X}}=L_P^{X}$. Since this holds for 
every $G$-Galois 
model 
$X\rightarrow Y$, it follows that $(K^{\oper{sep}})^{\alpha^{-1}(Z(G))}$ is contain in $L_P$. Therefore $\tilde X\rightarrow Y$ 
becomes 
Galois when base changed to $L_P$.

Fix a $G$-Galois model $X\rightarrow Y$ of $\bar X\rightarrow \bar Y$. Let 
$\mathbf{1}:\oper{Gal}(L_P^{X})\rightarrow G$ denote the 
map that sends all of 
$\oper{Gal}(L_P^{X})$ to $1$. Since the restriction of $\alpha$ to $L_P^{X}$ equals $\mathbf{1}$, the base change of 
$X^{\alpha}\rightarrow 
Y$ to $L_P^{X}$ is the twist $X^{\mathbf{1}}_{L_P^{X}}\rightarrow Y_{L_P^{X}}$ of $X_{L_P^{X}}\rightarrow Y_{L_P^{X}}$ by 
$\mathbf{1}$. This mere cover is clearly isomorphic to 
$X_{L_P^{X}}\rightarrow Y_{L_P^{X}}$ as a mere cover.
\end{proof}
\subsection{Appendix - Adjoining Roots of Unity to a Field of Moduli to get a Field of Definition}\label{E:adjoin}
Theorem \ref{E:tech} in Part I describes a general relationship between the field of
moduli and fields of definition. In this section we observe (Proposition
\ref{E:abelian}) the existence of a particular field of definition (infinite over the field of moduli)
with special properties. This allows us to prove a purely field-theoretic result (Corollary \ref{E:coolerigp}) 
towards the Inverse Galois Problem.

Let $G$ be a finite group, and let $\bar X\rightarrow \bar Y$ be a $G$-Galois branched cover of varieties over $\bar \q$.
As noted in Remark \ref{E:explanation}, its field of moduli $M$ as a $G$-Galois branched cover may not be a field of definition as a $G$-Galois branched
cover. However, Coombes
and Harbater
(\cite{ch})
proved that the field $\cup_nM(\zeta_n)$ resulting from adjoining all of the roots of unity to $M$ \it is \rm a field of definition.
(Here $\zeta_n$ is defined to be $e^{\frac{2\pi i}{n}}$.) The following is a
strengthening of this result.
\begin{prp} \label{E:abelian}
In the situation above, the field $\cup_{\{n|\exists
m:\,n|\,|Z(G)|^m\}}M(\zeta_n)$ is a field of definition. In particular, there
exists a field of definition (finite over $\q$) that is ramified over the field of moduli $M$ only
over the primes that divide $|Z(G)|$.
\end{prp}
\begin{proof}
If $G$ is centerless, then the cover is defined over its field of moduli
(\cite{ch}) and therefore the theorem follows. Otherwise $\cup_{\{n|\exists
m:\,n|\,|Z(G)|^m\}}M(\zeta_n)$ satisfies the
hypotheses of Proposition 9 in Chapter II of \cite{serregalois}. We conclude
that $\oper{cd}_p(\cup_{\{n|\exists
m:\,n|\,|Z(G)|^m\}}M(\zeta_n))\leq 1$ for every prime $p$ that divides
$|Z(G)|$. This implies that $H^2(\cup_{\{n|\exists
m:\,n|\,|Z(G)|^m\}}M(\zeta_n),Z(G))$ is trivial. As the obstruction for this
field to be a field of definition lies in this group (\cite{descent}), we are
done.
\end{proof}
We get the following corollary:

\begin{crl} \label{E:coolerigp} 
There is an extension of number
fields $\q\subset E\subset F$ such that $F/E$ is $G$-Galois, and $E/\q$
ramifies only over those primes that divide $|G|$.
\end{crl}
\begin{proof}
Let $r$ be the minimal number of generators of $G$. Let $p$ be a prime that divides $|G|$, and let $n$ be a natural number such that $r\leq p^n$. 
Let $T$ be the set of $(p^n)^{\oper{th}}$ roots of unity. Since $r\leq |T|$, a standard argument (see \cite{harbmsri}) using Riemann's Existence Theorem  shows that there exists a $G$-Galois branched cover
$\bar X\rightarrow \p_{\bar \q}$,
ramified at most over $T$. Let $M$ be its field of moduli as a $G$-Galois branched cover.

It is easy to check that the points of $T$, viewed as horizontal divisors in $\p_{\z}$,
coalesce only over the prime $p$. Therefore, by \cite{syb1}, the field extension $M/\q$ ramifies only over primes that divide $|G|$. By
Proposition \ref{E:abelian}, there exists a field of definition $E/M$ of $\bar X\rightarrow \p_{\bar \q}$ as a $G$-Galois branched cover
such that $E/M$ ramifies only over primes that divide $|G|$. Therefore $E/\q$ ramifies only over primes that divide $|G|$. The corollary
now follows by applying Hilbert's Irreducibility Theorem.
\end{proof}
\newpage
\bibliography{specializations}{}

\providecommand{\bysame}{\leavevmode\hbox to3em{\hrulefill}\thinspace}
\providecommand{\MR}{\relax\ifhmode\unskip\space\fi MR }
\providecommand{\MRhref}[2]{%
  \href{http://www.ams.org/mathscinet-getitem?mr=#1}{#2}
}
\providecommand{\href}[2]{#2}
\begin{thebibliography}{Mat84}

\bibitem[Bec86]{beckthesis}
Sybilla Beckmann, \emph{Fields of {D}efinition of {S}olvable {B}ranched
  {C}overings}, PhD Thesis, University of Pennsylvania (1986).

\bibitem[Bec88]{syb2}
\bysame, \emph{Galois groups of fields of definition of solvable branched
  coverings}, Compositio Math. \textbf{66} (1988), no.~2, 121--144.

\bibitem[Bec89]{syb1}
\bysame, \emph{Ramified primes in the field of moduli of branched coverings of
  curves}, J. Algebra \textbf{125} (1989), 236--255.

\bibitem[CH85]{ch}
Kevin Coombes and David Harbater, \emph{Hurwitz families and arithmetic
  {G}alois groups}, Duke Math.\ J. \textbf{52} (1985), no.~4, 821--839.

\bibitem[DD97]{vs}
Pierre D{\`e}bes and Jean-Claude Douai, \emph{Algebraic covers: field of moduli
  versus field of definition}, Ann. Sci. \'Ecole Norm. Sup. \textbf{(4) 30}
  (1997), no.~3, 303--338.

\bibitem[D{\`e}b99]{twist}
Pierre D{\`e}bes, \emph{Galois covers with prescribed fibers: the
  {B}eckmann-{B}lack problem}, Ann. Scuola Norm. Sup. Pisa, Cl. Sci. \textbf{4}
  (1999), no.~28, 273--286.

\bibitem[D{\`e}b02]{descent}
\bysame, \emph{Descent theory for algebraic covers}, Arithmetic fundamental
  groups and noncommutative algebra (Berkeley, CA), Proc. Sympos. Pure Math.,
  Amer. Math. Soc., Providence, RI \textbf{70} (2002), 3--25.

\bibitem[DG70]{dem}
Michel Demazure and Pierre Gabriel, \emph{Groupes alg\'ebriques. tome i:
  G\'eom\'etrie alg\'ebrique, g\'en\'eralit\'es, groupes commutatifs.}, Avec un
  appendice Corps de classes local par Michiel Hazewinkel. Masson \& Cie,
  \'Editeur, Paris; North-Holland Publishing Co., Amsterdam, 1970.

\bibitem[DG12]{debesgrunwald}
Pierre D{\`e}bes and Nour Ghazi, \emph{Galois covers and the
  {H}ilbert-{G}runwald property}, Ann. Inst. Fourier (Grenoble) \textbf{62}
  (2012), no.~3, 989--1013.

\bibitem[FJ08]{fj}
Michael Fried and Moshe Jarden, \emph{Field {A}rithmetic}, third ed.,
  Ergebnisse der Mathematik und ihrer Grenzgebiete. 3. Folge. A Series of
  Modern Surveys in Mathematics [Results in Mathematics and Related Areas. 3rd
  Series. A Series of Modern Surveys in Mathematics], 11., Springer-Verlag,
  Berlin, 2008.

\bibitem[Gro61]{sgaone}
Alexander Grothendieck, \emph{S\'eminaire de {G}\'eom\'etrie {A}lg\'ebrique},
  vol.~1, I.H.E.S, Paris, 1960-61.

\bibitem[Gro60]{grdesc}
\bysame, \emph{Technique de descente et th\'eor\`emes d'existence en
  g\'eom\'etrie alg\'ebrique. i. {D}escente par morphisme fid\`element plats},
  S\'eminaire Bourbaki (1960).

\bibitem[Har03]{harbmsri}
David Harbater, \emph{Patching and {G}alois theory}, Math. Sci. Res. Inst.
  Publ. \textbf{41} (2003), 313--424.

\bibitem[Mat84]{matzat}
B.~Heinrich Matzat, \emph{Konstruktion von {Z}ahl- und {F}unktionenk\"{o}rpern
  mit vorgegebener {G}aloisgruppe}, J. reine u. angew. Math. \textbf{349}
  (1984), 179--220.

\bibitem[Ray90]{raynaud}
Michel Raynaud, \emph{$p$-groupes et r\'eduction semi-stable des courbes}, The
  Grothendieck Festschrift \textbf{III} (1990), 179--197.

\bibitem[Ser94]{serregalois}
Jean-Pierre Serre, \emph{Cohomologie {G}aloisienne}, Lecture Notes in
  Mathematics, no.~5, Springer-Verlag, Berlin, 1994.

\bibitem[Sko01]{torsy}
Alexei Skorobogatov, \emph{Torsors and {R}ational {P}oints}, Cambridge Tracts
  in Mathematics, Cambridge University Press, Cambridge, 2001.

\bibitem[Wei82]{weissauer}
Rainer Weissauer, \emph{Der {H}ilbertsche {I}rreduzibilit\"atssatz}, J. Reine
  Angew. Math \textbf{334} (1982), 203--220.

\end{thebibliography}
\bibliographystyle{amsalpha}
\medskip
\noindent Current author information:\\
Hilaf Hasson: Department of Mathematics, Stanford University, Palo Alto, CA 94305, USA\\
email: {\tt hilaf@stanford.edu}
\end{document}